\documentclass[12pt]{amsart}
\usepackage{amssymb,amsmath,amsthm,mathrsfs,multirow,xcolor,framed,url}
\usepackage[pdftex,
         pdfauthor={Dandan Chen and Siyu Yin},
         pdftitle={Congruences modulo powers of $3$ for generalized Frobenius partitions $C\Psi_{6,0}$},
         pdfsubject={MATHEMATICS},
         pdfkeywords={Generalized Frobenius partitions; Atkin-Lehner involution; Congruences; Modular forms; Partitions},
         pdfproducer={Latex with hyperref},
         pdfcreator={pdflatex}]{hyperref}
\usepackage [latin1]{inputenc}
\newtheorem{theorem}{Theorem}[section]
\newtheorem{lemma}[theorem]{Lemma}

\theoremstyle{definition}
\newtheorem{definition}[theorem]{Definition}

\theoremstyle{remark}
\newtheorem{remark}[theorem]{Remark}

\numberwithin{equation}{section}

\newcommand\nutwid{\overset {\text{\lower 3pt\hbox{$\sim$}}}\nu}

\newcommand{\SL}{\mbox{SL}}

\allowdisplaybreaks

\newcommand\omycite[1]{}

\newcommand{\beqs}{\begin{equation*}}
\newcommand{\eeqs}{\end{equation*}}
\newcommand{\beq}{\begin{equation}}
\newcommand{\eeq}{\end{equation}}
\renewcommand{\MR}[1]{\href{http://www.ams.org/mathscinet-getitem?mr={#1}}{MR{#1}}}

\begin{document}
\title[Congruences modulo powers of $3$ ]{Congruences modulo powers of $3$ for generalized Frobenius partitions $C\Psi_{6,0}$ }


\author{Dandan Chen}
\address{Department of Mathematics, Shanghai University, People's Republic of China}
\address{Newtouch Center for Mathematics of Shanghai University, Shanghai, People's Republic of China}
\email{mathcdd@shu.edu.cn}
\author{Siyu Yin*}
\address{Department of Mathematics, Shanghai University, People's Republic of China}
\email{siyuyin@shu.edu.cn, siyuyin0113@126.com}



\date{}



\begin{abstract}
In 1984, Andrews introduced the family of partition functions \(c\phi_k(n)\), which counts the number of generalized Frobenius partitions of \(n\) with \(k\) colors. In previous work, we proved a conjecture on congruences for \(c\phi_6(n)\) modulo powers of 3. In this paper, we consider the \((6,0)\)-colored Frobenius partition functions \(c\psi_{6,0}(n)\). We establish a connection between the generating functions of \(c\psi_{6,3}(n)\) and \(c\psi_{6,0}(n)\) via an Atkin-Lehner involution, and prove congruences modulo powers of 3 for \(c\psi_{6,0}(n)\).
\end{abstract}

\maketitle


\section{Introduction}

In his 1984 AMS Memoir \cite{Andrews-1984}, Andrews introduced the concept of generalized Frobenius partitions of \(n\). Such a partition is represented by a two-row array
\[
\begin{pmatrix}
a_1 & a_2 & \cdots & a_r \\
b_1 & b_2 & \cdots & b_r
\end{pmatrix},
\]
where the \(a_i\) and \(b_i\) are non-negative integers satisfying
\[
n = r + \sum_{i=1}^r a_i + \sum_{i=1}^r b_i,
\]
and each row is arranged in non-increasing order.

Andrews \cite{Andrews-1984} also discussed a variant of generalized Frobenius partitions, referred to as a $k$-colored generalized Frobenius partition. A $k$-colored generalized Frobenius partition is an array of the above form in which the integer entries are taken from $k$ distinct copies of the non-negative integers, distinguished by color. The rows are ordered first by size and then by color, with the condition that no two consecutive entries in any row share the same color.

In 2022, Jiang, Rolen, and Woodbury \cite{Jiang-Rolen-Woodbury-2022} introduced the $(k,a)$-colored Frobenius partition functions $c\psi_{k,a}(n)$, which are natural generalizations of Andrews' $c\phi_k(n)$ since $c\psi_{k,\frac{k}{2}}(n)=c\phi_k(n)$. They showed that for $a \in \frac{k}{2} + \mathbb{Z}$, the generating function of $c\psi_{k,a}(n)$,
\[
C\Psi_{k,a}(q) := \sum_{n=0}^{\infty} c\psi_{k,a}(n) q^n,
\]
is the $\zeta^a$ coefficient of
\[
F_k(\tau,z) = \left( \frac{-\theta(\tau, z + \frac{1}{2})}{q^{\frac{1}{12}} \eta(\tau)} \right)^{\!k},
\]
where $q := e^{2\pi i \tau}$, $\zeta := e^{2\pi i z}$, $\tau \in \mathbb{H}$ (the upper half-plane), $z \in \mathbb{C}$, and $\eta(\tau) := q^{\frac{1}{24}} (q;q)_{\infty}$.

The first author, together with  Chen and Garvan \cite{Chen2-Garvan-arxiv}, showed how congruences modulo powers of $5$ and $7$ for the coefficients of the third-order mock theta function $f(q)$ imply congruences modulo powers of $5$ and $7$ for the coefficients of the related third-order mock theta function $\omega(q)$, using Atkin--Lehner involutions and transformation results of Zwegers. Subsequently, Garvan, Sellers, and Smoot \cite{Garvan-Sellers-Smoot-2024} demonstrated how Atkin--Lehner involutions imply surprising connections for families of Frobenius partition congruences for the generalized Frobenius partitions $c\psi_{2,0}(n)$ and $c\psi_{2,1}(n)$. They also emphasized that considerations for the general case of $c\psi_{k,\beta}$ are important for future work. For each $k$, Chen and Zhu \cite{Chen-Zhu-arxiv} constructed a vector-valued modular form for the generating functions of $c\psi_{k,\beta}(n)$ and determined an equivalence relation among all $\beta$. With the aid of these correspondences, they proved a family of congruences for $c\phi_3(n)$. Wang and Wang \cite{Wang-Wang} gave a new example of Chen and Zhu's framework for $c\psi_{4,\beta}(n)$.

In this paper, we establish a connection between $C\Psi_{6,3}$ and $C\Psi_{6,0}$ via an Atkin--Lehner involution, and prove congruences modulo powers of $3$ for $c\psi_{6,0}(n)$. From \cite[Theorem 2.1]{Baruah-Sarmah-2015}, Baruah and Sarmah expressed the generating function of $c\phi_6(n)$ in terms of Ramanujan's theta functions:
\[
\begin{aligned}
\sum_{n=0}^{\infty} c\phi_6(n) q^n = \frac{1}{(q;q)_{\infty}^6} \bigl\{ & \varphi^3(q) \varphi(q^2) \varphi(q^6) + 24q \psi^3(q) \psi(q^2) \psi(q^3) \\
& + 4q^2 \varphi^3(q) \psi(q^4) \psi(q^{12}) \bigr\},
\end{aligned}
\]
where
\[
\varphi(q) = (-q;q^2)_{\infty}^2 (q^2;q^2)_{\infty}, \qquad
\psi(q) = (-q;q^2)_{\infty} (q^4;q^4)_{\infty}.
\]
Applying \cite[Theorem 2]{Jiang-Rolen-Woodbury-2022}, we derive the generating function of $c\psi_{6,0}(n)$, which is presented in Lemma \ref{generating-cpsi-6-0}.

We have previously established the following congruence family for the generalized $6$-colored Frobenius partition function $c\phi_6(n)$ (i.e., $c\psi_{6,3}(n)$):

\begin{theorem}\label{cpsi-6-3}
\cite[Theorem 1.1]{Chen-Yin-arxiv}
Let $n, \alpha \in \mathbb{Z}_{\geq 1}$ such that $4n \equiv 1 \pmod{3^\alpha}$. Then
\[
c\phi_6(n) \equiv 0 \pmod{3^{\lfloor \alpha/2 \rfloor + 2}}.
\]
\end{theorem}

Inspired by \cite{Chen2-Garvan-arxiv},\cite{Chen-Zhu-arxiv} and \cite{Garvan-Sellers-Smoot-2024}, we connect $C\Psi_{6,0}$ with $C\Psi_{6,3}$ via Atkin--Lehner involutions and discover that $c\psi_{6,0}(n)$ exhibits the following congruence family:

\begin{theorem}\label{cpsi-6-0}
Let $n, \alpha \in \mathbb{Z}_{\geq 1}$ such that $2n \equiv -1 \pmod{3^\alpha}$. Then
\[
c\psi_{6,0}(n) \equiv 0 \pmod{3^{\lfloor \alpha/2 \rfloor + 2}}.
\]
\end{theorem}

In this paper, we prove Theorem \ref{cpsi-6-0} by using Atkin--Lehner involutions to relate $C\Psi_{6,3}$ and $C\Psi_{6,0}$. The paper is organized as follows. Section \ref{sec-pre} collects the necessary notations and definitions. In Section \ref{sec-cpsi-6-0}, we briefly recall the proof of the congruence properties for $c\phi_6(n)$ modulo powers of $3$ and then use Atkin--Lehner involutions to prove Theorem \ref{cpsi-6-0}.

\section{Preliminaries}\label{sec-pre}

Throughout the paper, we use the following conventions: $\mathbb{N}^* = \{1,2,\dots\}$ denotes the set of positive integers. For $x \in \mathbb{R}$, the symbol $\lfloor x \rfloor$ denotes the greatest integer less than or equal to $x$.

\begin{definition}
For $f: \mathbb{H} \to \mathbb{C}$ and $m \in \mathbb{N}^*$, we define $U_m(f): \mathbb{H} \to \mathbb{C}$ by
\[
U_m(f)(\tau) := \frac{1}{m} \sum_{\lambda=0}^{m-1} f\!\left(\frac{\tau+\lambda}{m}\right), \qquad \tau \in \mathbb{H}.
\]
\end{definition}

We also write $U_m(f)$ as $f|U_m$. The operator $U_m$ is linear over $\mathbb{C}$. The operators $U_m$, introduced by Atkin and Lehner \cite{Atkin-Lehner-1970}, are closely related to Hecke operators. They typically arise in the context of partition congruences \cite{Andrews-1976}, mostly because of the property that if
\[
f(\tau) = \sum_{n=-\infty}^{\infty} f_n q^n \qquad (q = e^{2\pi i \tau}),
\]
then
\[
U_m(f)(\tau) = \sum_{n=-\infty}^{\infty} f_{mn} q^n.
\]

We also recall the Atkin--Lehner involutions and their actions on modular functions, following \cite[Lemma 7]{Atkin-Lehner-1970} and \cite[Lemma 5.2]{Chen2-Garvan-arxiv}.

\begin{lemma}
Let $p$ be prime. If $f$ is a non-holomorphic modular function on $\Gamma_0(pN)$ with $p \mid N$, then $f \mid U_p$ is a non-holomorphic modular function on $\Gamma_0(N)$.
\end{lemma}

\begin{definition}\cite[p.~3]{Chan-Lang-1998}\label{W-operator}
If $e \parallel N$, we call the matrix
\[
W_e := \begin{pmatrix} ae & b \\ cN & de \end{pmatrix}, \qquad a,b,c,d \in \mathbb{Z}, \quad \det(W_e) = e,
\]
an Atkin--Lehner involution of $\Gamma_0(N)$.
\end{definition}

\begin{lemma}\cite[Lemma 6]{Chan-Toh}
Let $p$ be prime, $p \mid N$, $e \parallel N$, and $(p,e)=1$. If $f(\tau)$ is a modular function on $\Gamma_0(N)$, then
\[
(f \mid U_p) \mid W_e = (f \mid W_e) \mid U_p.
\]
\end{lemma}

\begin{lemma}\label{W-effect}\cite[Corollary 2.2]{Chan-Lang-1998}
Let $W_e$ be an Atkin--Lehner involution of $\Gamma_0(N)$. Let $t > 0$ be such that $t \mid N$. Then
\[
\eta(t W_e \tau) = \eta\!\left(t \frac{ae\tau + b}{cN\tau + de}\right) = v_{\eta}(M) \left(\frac{cN\tau + de}{\delta}\right)^{\!1/2} \eta\!\left(\frac{et}{\delta} \tau\right),
\]
where $\delta = (e,t)$, $v_\eta$ is the eta-multiplier, and
\[
M = \begin{pmatrix} a\delta & bt/\delta \\ cN\delta/et & de/\delta \end{pmatrix} \in \SL_2(\mathbb{Z}).
\]
\end{lemma}

The proof of Theorem \ref{cpsi-6-0} relies on some identities among eta-quotients. These identities can be verified using the theory of modular functions. By making use of the valence formula for modular forms, Garvan has written a MAPLE package called ETA; see
\[
\label{r:eta}
{https://qseries.org/fgarvan/qmaple/ETA/}.
\]
A tutorial for this package can be found in \cite{gtutorial}.

We need the following generating function of $c\psi_{6,0}(n)$.

\begin{lemma}\label{generating-cpsi-6-0}
We have
\begin{align*}
\sum_{n=0}^{\infty} c\psi_{6,0}(n) q^n = q^{-\frac{1}{2}} \Bigg( & 8 \frac{\eta_4^{11} \eta_{12}^5}{\eta_2^5 \eta_1^6 \eta_6^2 \eta_8^2 \eta_{24}^2} + 4 \frac{\eta_2^9 \eta_8^2 \eta_{12}^5}{\eta_1^{10} \eta_4^3 \eta_6^2 \eta_{24}^2} \\
& + 32 \frac{\eta_4^5 \eta_8^2 \eta_{24}^2}{\eta_1^6 \eta_2^3 \eta_{12}} + 4 \frac{\eta_4^3 \eta_2^7 \eta_{24}^2}{\eta_1^{10} \eta_8^2 \eta_{12}} + 8 \frac{\eta_2^{11} \eta_6^2}{\eta_1^{11} \eta_4^2 \eta_3} \Bigg).
\end{align*}
\end{lemma}

\begin{proof}
Using Theorem 2 of \cite{Jiang-Rolen-Woodbury-2022}, we have
\[
\sum_{n=0}^{\infty} c\psi_{6,0}(n) q^n = \frac{h_{3,0}}{q^{\frac{1}{2}} \eta_1^6},
\]
where
\begin{align*}
h_{3,0} &= \theta_{1,1} \theta_{6,0} h_{2,0} + 2 \theta_{1,0} \theta_{6,3} h_{2,1} + \theta_{1,1} \theta_{6,6} h_{2,2},\\
h_{2,0} &= \theta_{1,1}^2 \theta_{2,0} + \theta_{1,0}^2 \theta_{2,2},\\
h_{2,1} &= 2 \theta_{1,0} \theta_{1,1} \theta_{2,1},\\
h_{2,2} &= \theta_{1,1}^2 \theta_{2,2} + \theta_{1,0}^2 \theta_{2,0},\\
\theta_{m,a} &= \sum_{n \in \mathbb{Z}} q^{\frac{(2mn + a)^2}{4m}}.
\end{align*}
Recalling three well-known consequences of the Jacobi triple product identity,
\[
\theta_{1,0} = \frac{\eta_2^5}{\eta_1^2 \eta_4^2}, \qquad
\theta_{1,1} = \frac{2 \eta_4^2}{\eta_2}, \qquad
\theta_{2,1} = \frac{\eta_2^2}{\eta_1}.
\]
After tedious but elementary simplification, we obtain the generating function of $c\psi_{6,0}(n)$.
\end{proof}

We now review the notations employed in the proof of Theorem \ref{cpsi-6-3} in \cite{Chen-Yin-arxiv}.

\begin{definition}\cite[Definition 2.2]{Chen-Yin-arxiv}\label{def-A-B-3}
For $f: \mathbb{H} \to \mathbb{C}$, we define $U_A(f), U_B(f): \mathbb{H} \to \mathbb{C}$ by $U_A(f) := U_3(A f)$ and $U_B(f) := U_3(B f)$, where
\[
A := \frac{\eta_9^9 \eta_4^2 \eta_2^5}{\eta_{36}^2 \eta_{18}^5 \eta_1^9} \qquad \text{and} \qquad B := \frac{\eta_9 \eta_2^2}{\eta_{18}^2 \eta_1}.
\]
\end{definition}

\begin{definition}\cite[Definition 2.4]{Chen-Yin-arxiv}
\label{r:defL}
We define the $U$-sequence $(L_{\alpha})_{\alpha \geq 0}$ by
\[
L_0 := \frac{\eta_{12}^5 \eta_3 \eta_2^8}{\eta_{24}^2 \eta_8^2 \eta_6^4 \eta_4^3 \eta_1^3} + 24 + 4 \frac{\eta_{24}^2 \eta_8^2 \eta_3 \eta_2^{10}}{\eta_{12} \eta_6^2 \eta_4^9 \eta_1^3},
\]
and for $\alpha \geq 1$:
\[
L_{2\alpha-1} := U_A(L_{2\alpha-2}) \qquad \text{and} \qquad L_{2\alpha} := U_B(L_{2\alpha-1}).
\]
\end{definition}

\begin{definition}\label{t-p-y}\cite[Definition 2.6]{Chen-Yin-arxiv}
Let $t$, $y$, $p_0$, $p_1$ be functions defined on $\mathbb{H}$ as follows:
\[
t := \frac{\eta_{12}^4 \eta_2^2}{\eta_6^2 \eta_4^4}, \qquad
y := \frac{\eta_4^3 \eta_3}{\eta_{12} \eta_1^3}, \qquad
p_0 := \frac{\eta_{12}^4 \eta_3^{12} \eta_2^8}{\eta_6^8 \eta_4^{12} \eta_1^4}, \qquad
p_1 := \frac{\eta_{12}^2 \eta_3^6 \eta_2^4}{\eta_6^4 \eta_4^6 \eta_1^2},
\]
which have Laurent series expansions in powers of $q$ with coefficients in $\mathbb{Z}$.
\end{definition}

Besides, using Garvan's ETA package, we obtain
\[
L_0 = t^{-1} + 9t^2 + 3t + 27.
\]

\section{Proof of Theorem \ref{cpsi-6-0}}\label{sec-cpsi-6-0}

From \cite[p.~11]{Chen-Yin-arxiv}, we know that each $L_{\alpha}$ can be written as
\begin{align}
L_{2\alpha-1} &= p_0 \, y^{3^{2\alpha}-1} \sum_{n \geq -1} d_n^{(2\alpha-1)} t^n; \label{L-ODD}\\
L_{2\alpha} &= p_1 \, y^{3^{2\alpha+1}-3} \sum_{n \geq 0} d_n^{(2\alpha)} t^n, \label{L-EVEN}
\end{align}
for $\alpha \geq 1$. Theorem \ref{cpsi-6-3} was proved by showing that $L_{\alpha}$ is divisible by $3^{\lfloor \alpha/2 \rfloor + 2}$. We attempt to establish a connection between $C\Psi_{6,3}$ and $C\Psi_{6,0}$. Inspired by \cite{Garvan-Sellers-Smoot-2024}, we construct an explicit mapping between these generating functions by first sending $q \mapsto -q$, then applying an Atkin-Lehner involution matrix $W_4$, and then sending $q \mapsto -q$ once more. In particular, we consider the operator
\[
W_4 := \begin{pmatrix} 28 & 3 \\ 36 & 4 \end{pmatrix}.
\]
We begin by taking the function $A$ which is necessary for the construction of our $U_A$ operators. For $A$, we first map $q \mapsto -q$, obtaining
\[
A \mapsto \frac{\eta_1^9 \eta_4^{11} \eta_{18}^{22}}{\eta_2^{22} \eta_9^9 \eta_{36}^{11}}.
\]
We then send $\tau$ to $W\tau$. By Lemma \ref{W-effect}, we obtain
\[
\frac{\eta_1^9 \eta_4^{11} \eta_{18}^{22}}{\eta_2^{22} \eta_9^9 \eta_{36}^{11}} \mapsto \frac{\eta_1^{11} \eta_4^9 \eta_{18}^{22}}{\eta_2^{22} \eta_9^{11} \eta_{36}^9}.
\]
Finally, applying $q \mapsto -q$, we arrive at
\[
\frac{\eta_1^{11} \eta_4^9 \eta_{18}^{22}}{\eta_2^{22} \eta_9^{11} \eta_{36}^9} \mapsto \frac{\eta_2^{11} \eta_9^{11} \eta_{36}^2}{\eta_1^{11} \eta_4^2 \eta_{18}^{11}} =: \widetilde{A}.
\]
Composing the three successive mappings, we find that the matrix representing the overall action is
\[
\begin{pmatrix} 1 & 1/2 \\ 0 & 1 \end{pmatrix}
\begin{pmatrix} 28 & 3 \\ 36 & 4 \end{pmatrix}
\begin{pmatrix} 1 & 1/2 \\ 0 & 1 \end{pmatrix}
= \begin{pmatrix} 46 & 28 \\ 36 & 22 \end{pmatrix},
\]
which reduces via the associated transformation to
\[
\gamma = \begin{pmatrix} 23 & 14 \\ 18 & 11 \end{pmatrix} \in \Gamma_0(18).
\]
Under the action of $\gamma$, the auxiliary functions introduced earlier transform as follows:
\begin{align}
\widetilde{B} &:= B(\gamma\tau) = \frac{\eta_2^2 \eta_9}{\eta_1 \eta_{18}^2}, \nonumber \\
\widetilde{t} &:= t(\gamma\tau) = \frac{\eta_1^4 \eta_4^4 \eta_6^{10}}{\eta_2^{10} \eta_3^4 \eta_{12}^4}, \label{t}\\
\widetilde{y} &:= y(\gamma\tau) = -\frac{\eta_{12} \eta_3^2 \eta_2^9}{2 \eta_6^3 \eta_4^3 \eta_1^6}, \label{y}\\
\widetilde{L_0} &:= L_0(\gamma\tau) = \widetilde{t}^{-1} + 9\widetilde{t}^2 + 3\widetilde{t} + 27, \label{L0}\\
\widetilde{p_0} &:= p_0(\gamma\tau) = (1+t)^4(\gamma\tau) = (1+\widetilde{t})^4 = \frac{16 \eta_6^4 \eta_4^{12} \eta_3^8 \eta_1^{8}}{\eta_{12}^4 \eta_2^{28}}, \label{p0}\\
\widetilde{p_1} &:= p_1(\gamma\tau) = (1+t)^2(\gamma\tau) = (1+\widetilde{t})^2 = \frac{4 \eta_6^2 \eta_4^{6} \eta_3^4 \eta_1^{4}}{\eta_{12}^2 \eta_2^{14}}, \label{p1}
\end{align}
where the last equalities in \eqref{p0} and \eqref{p1} are verified directly using the ETA Maple package.

By Lemma \ref{generating-cpsi-6-0} and \eqref{L0}, using the ETA Maple package, we obtain
\[
\widetilde{L_0} = \frac{2 q^{1/2} \eta_3 \eta_1^{11} \eta_4^2}{\eta_6^2 \eta_2^{11}} \, C\Psi_{6,0}(q).
\]

Guided by the proof of Theorem \ref{cpsi-6-3}, we now define the $U$-sequence $(\widetilde{L}_{\alpha})_{\alpha \geq 0}$ for $\alpha \geq 1$:
\[
\widetilde{L}_{2\alpha-1} := U_{\widetilde{A}}(\widetilde{L}_{2\alpha-2}) \qquad \text{and} \qquad \widetilde{L}_{2\alpha} := U_{\widetilde{B}}(\widetilde{L}_{2\alpha-1}).
\]
Then we obtain the following lemma. Its proof is completely analogous to \cite[p.~23]{Atkin-1967} and is therefore omitted.

\begin{lemma}\label{lem-L-alpha}
For $\alpha \in \mathbb{N}^*$, we have
\begin{align*}
\widetilde{L}_{2\alpha-1} &= 2 \prod_{n=1}^{\infty} \frac{(1-q^n)(1-q^{3n})^{11}(1-q^{12n})^2}{(1-q^{2n})^2(1-q^{6n})^{11}} \sum_{n=0}^{\infty} c\psi_{6,0}(3^{2\alpha-1}n + \lambda_{2\alpha-1}) q^n,\\
\widetilde{L}_{2\alpha} &= 2 \prod_{n=1}^{\infty} \frac{(1-q^{n})^{11}(1-q^{4n})^2(1-q^{3n})}{(1-q^{2n})^{11}(1-q^{6n})^2} \sum_{n=0}^{\infty} c\psi_{6,0}(3^{2\alpha}n + \lambda_{2\alpha}) q^n.
\end{align*}
\end{lemma}

Finally, we prove Theorem \ref{cpsi-6-0}.

\begin{proof}
We first observe that
\[
\widetilde{L}_1 = U_3(\widetilde{A}) = U_3(A(\gamma\tau)) = L_1(\gamma\tau).
\]
Assume inductively that for some $\alpha \geq 2$ we have $\widetilde{L}_{\alpha-1} = L_{\alpha-1}(\gamma\tau)$. Without loss of generality, take $\alpha$ even. Then
\begin{align*}
\widetilde{L}_{\alpha} &= U_{\widetilde{B}}(\widetilde{L}_{\alpha-1}) \\
&= U_3(\widetilde{B} \widetilde{L}_{\alpha-1}) \\
&= U_3(B(\gamma\tau) L_{\alpha-1}(\gamma\tau)) \\
&= U_3(B L_{\alpha-1})(\gamma\tau) \\
&= L_{\alpha}(\gamma\tau).
\end{align*}

Hence $\widetilde{L}_{\alpha} = L_{\alpha}(\gamma\tau)$ for every $\alpha \in \mathbb{N}$. Using \eqref{L-ODD}, \eqref{L-EVEN}, \eqref{t}, \eqref{y}, \eqref{p0}, and \eqref{p1}, we obtain that
\begin{align*}
\widetilde{L}_{2\alpha-1} &= \widetilde{p_0} \, \widetilde{y}^{\,3^{2\alpha}-1} \sum_{n \geq -1} d_n^{(2\alpha-1)} \widetilde{t}^{\,n},\\
\widetilde{L}_{2\alpha} &= \widetilde{p_1} \, \widetilde{y}^{\,3^{2\alpha+1}-3} \sum_{n \geq 0} d_n^{(2\alpha)} \widetilde{t}^{\,n},
\end{align*}
for $\alpha \geq 1$. The sequences $L_{\alpha}$ and $\widetilde{L}_{\alpha}$ have the same coefficients, respectively. Therefore, $\widetilde{L}_{\alpha} = L_{\alpha}(\gamma\tau)$ retains the same divisibility properties, and Theorem \ref{cpsi-6-0} follows immediately.
\end{proof}

\begin{remark}
Equation \eqref{y} shows that $2\widetilde{y}$ has a Laurent series expansion in powers of $q$ with integer coefficients. Since $2$ is coprime to $3$, the factor $\widetilde{y}$ does not affect the divisibility by $3$ of $c\psi_{6,0}(n)$. Moreover, $\widetilde{L_0}$, $\widetilde{A}$, and $\widetilde{B}$ all have Laurent series expansions in powers of $q$ with integer coefficients; consequently, $\widetilde{L}_{\alpha}$ also has such an expansion for every $\alpha \in \mathbb{N}$.
\end{remark}

\subsection*{Acknowledgements}
The first author was  supported by the National Key R\&D Program of China (Grant No. 2024YFA1014500) and the National Natural Science Foundation of China (Grant No. 12201387).






\end{document}